\documentclass[12pt]{amsart}
\usepackage{amsmath,amssymb,amsbsy,amsfonts,latexsym,amsopn,amstext,cite,
                                               amsxtra,euscript,amscd,bm}
\usepackage{url}

\usepackage{mathrsfs}

\usepackage{color}
\usepackage[colorlinks,linkcolor=blue,anchorcolor=blue,citecolor=blue,backref=page]{hyperref}
\usepackage{color}
\usepackage{graphics,epsfig}
\usepackage{graphicx}
\usepackage{float}
\usepackage{epstopdf}
\hypersetup{breaklinks=true}

\usepackage[np]{numprint}
\npdecimalsign{\ensuremath{.}}

\usepackage{bibentry}

\usepackage[english]{babel}
\usepackage{mathtools}
\usepackage{todonotes}

\usepackage[colorlinks,linkcolor=blue,anchorcolor=blue,citecolor=blue,backref=page]{hyperref}

\usepackage[norefs,nocites]{refcheck}

\DeclareMathOperator*\sign{sign}

\begin{document}

\newcommand{\bs}{\boldsymbol}
\def \a{\alpha} \def \b{\beta} \def \d{\delta} \def \e{\varepsilon} \def \g{\gamma} \def \k{\kappa} \def \l{\lambda} \def \s{\sigma} \def \t{\theta} \def \z{\zeta}

\newcommand{\mb}{\mathbb}

\newtheorem{theorem}{Theorem}
\newtheorem{lemma}[theorem]{Lemma}
\newtheorem{claim}[theorem]{Claim}
\newtheorem{cor}[theorem]{Corollary}
\newtheorem{conj}[theorem]{Conjecture}
\newtheorem{prop}[theorem]{Proposition}
\newtheorem{definition}[theorem]{Definition}
\newtheorem{question}[theorem]{Question}
\newtheorem{example}[theorem]{Example}
\newcommand{\hh}{{{\mathrm h}}}
\newtheorem{remark}[theorem]{Remark}
\newtheorem{prob}[theorem]{Problem}

\numberwithin{equation}{section}
\numberwithin{theorem}{section}
\numberwithin{table}{section}
\numberwithin{figure}{section}

\def\sssum{\mathop{\sum\!\sum\!\sum}}
\def\ssum{\mathop{\sum\ldots \sum}}
\def\iint{\mathop{\int\ldots \int}}

\newcommand{\diam}{\operatorname{diam}}
\newcommand{\chr}{\operatorname{char}}

\def\squareforqed{\hbox{\rlap{$\sqcap$}$\sqcup$}}
\def\qed{\ifmmode\squareforqed\else{\unskip\nobreak\hfil
\penalty50\hskip1em \nobreak\hfil\squareforqed
\parfillskip=0pt\finalhyphendemerits=0\endgraf}\fi}%%

%  use the AMS-Euler Fraktur fonts
%%%%%%%%%%%%%%%%%%%%%%%%%%%%%%%%%%
\newfont{\teneufm}{eufm10}
\newfont{\seveneufm}{eufm7}
\newfont{\fiveeufm}{eufm5}
%%%%%%%%%%%%%%%%%%%%%%%%%%%%%%%%%
%
%  allow automatic size selection in math mode
%
%%%%%%%%%%%%%%%%%%%%%%%%%%%%%%%%%
\newfam\eufmfam
     \textfont\eufmfam=\teneufm
\scriptfont\eufmfam=\seveneufm
     \scriptscriptfont\eufmfam=\fiveeufm
%%%%%%%%%%%%%%%%%%%%%%%%%%%%%%%%%
%
%  \frak works on a single symbol at a time...
%
\def\frak#1{{\fam\eufmfam\relax#1}}

\newcommand{\bflambda}{{\boldsymbol{\lambda}}}
\newcommand{\bfmu}{{\boldsymbol{\mu}}}
\newcommand{\bfxi}{{\boldsymbol{\eta}}}
\newcommand{\bfrho}{{\boldsymbol{\rho}}}

\def\eps{\varepsilon}

\def\fK{\mathfrak K}
\def\fT{\mathfrak{T}}
\def\fL{\mathfrak L}
\def\fR{\mathfrak R}
\def\fQ{\mathfrak Q}

\def\fA{{\mathfrak A}}
 \def\fB{{\mathfrak B}}
\def\fC{{\mathfrak C}}
\def\fL{{\mathfrak L}}
\def\fM{{\mathfrak M}}
\def\fS{{\mathfrak  S}}
\def\fU{{\mathfrak U}}

\def\sssum{\mathop{\sum\!\sum\!\sum}}
\def\ssum{\mathop{\sum\ldots \sum}}
\def\dsum{\mathop{\quad \sum \qquad \sum}}
\def\iint{\mathop{\int\ldots \int}}
 
\def\T {\mathsf {T}}
\def\Tor{\mathsf{T}_d}
\def\Tore{\widetilde{\mathrm{T}}_{d} }

\def\sM {\mathsf {M}}
\def\sL {\mathsf {L}}
\def\sK {\mathsf {K}}
\def\sP {\mathsf {P}}

\def\ss{\mathsf {s}}

\def \balpha{\bm{\alpha}}
\def \bbeta{\bm{\beta}}
\def \bgamma{\bm{\gamma}}
\def \bdelta{\bm{\delta}}
\def \bzeta{\bm{\zeta}}
\def \blambda{\bm{\lambda}}
\def \bchi{\bm{\chi}}
\def \bphi{\bm{\varphi}}
\def \bpsi{\bm{\psi}}
\def \bxi{\bm{\xi}}
\def \bmu{\bm{\mu}}
\def \bnu{\bm{\nu}}
\def \bomega{\bm{\omega}}

\def \bell{\bm{\ell}}

\def\eqref#1{(\ref{#1})}

\def\vec#1{\mathbf{#1}}

\newcommand{\abs}[1]{\left| #1 \right|}

\def\Zq{\mathbb{Z}_q}
\def\Zqx{\mathbb{Z}_q^*}
\def\Zd{\mathbb{Z}_d}
\def\Zdx{\mathbb{Z}_d^*}
\def\Zf{\mathbb{Z}_f}
\def\Zfx{\mathbb{Z}_f^*}
\def\Zp{\mathbb{Z}_p}
\def\Zpx{\mathbb{Z}_p^*}
\def\cM{\mathcal M}
\def\cE{\mathcal E}
\def\cH{\mathcal H}

\def\le{\leqslant}
\def\leq{\leqslant}
\def\ge{\geqslant}
\def\leq{\leqslant}

\def\sfB{\mathsf {B}}
\def\sfC{\mathsf {C}}
\def\sfS{\mathsf {S}}
\def\sfI{\mathsf {I}}
\def\sfT{\mathsf {T}}
\def\L{\mathsf {L}}
\def\FF{\mathsf {F}}

\def\sE {\mathscr{E}}
\def\sS {\mathscr{S}}

%%%%%%%%%%%%%%%%%%%%%%%%%
% Alphabet calligraphie %
%%%%%%%%%%%%%%%%%%%%%%%%%
\def\cA{{\mathcal A}}
\def\cB{{\mathcal B}}
\def\cC{{\mathcal C}}
\def\cD{{\mathcal D}}
\def\cE{{\mathcal E}}
\def\cF{{\mathcal F}}
\def\cG{{\mathcal G}}
\def\cH{{\mathcal H}}
\def\cI{{\mathcal I}}
\def\cJ{{\mathcal J}}
\def\cK{{\mathcal K}}
\def\cL{{\mathcal L}}
\def\cM{{\mathcal M}}
\def\cN{{\mathcal N}}
\def\cO{{\mathcal O}}
\def\cP{{\mathcal P}}
\def\cQ{{\mathcal Q}}
\def\cR{{\mathcal R}}
\def\cS{{\mathcal S}}
\def\cT{{\mathcal T}}
\def\cU{{\mathcal U}}
\def\cV{{\mathcal V}}
\def\cW{{\mathcal W}}
\def\cX{{\mathcal X}}
\def\cY{{\mathcal Y}}
\def\cZ{{\mathcal Z}}
\newcommand{\rmod}[1]{\: \mbox{mod} \: #1}

\def\cg{{\mathcal g}}

\def\vX{\mathbf X}
\def\vY{\mathbf Y}

\def\vy{\mathbf y}
\def\vr{\mathbf r}
\def\vx{\mathbf x}
\def\va{\mathbf a}
\def\vb{\mathbf b}
\def\vc{\mathbf c}
\def\vd{\mathbf d}
\def\ve{\mathbf e}
\def\vf{\mathbf f}
\def\vg{\mathbf g}
\def\vh{\mathbf h}
\def\vk{\mathbf k}
\def\vm{\mathbf m}
\def\vz{\mathbf z}
\def\vu{\mathbf u}
\def\vv{\mathbf v}

\def\e{{\mathbf{\,e}}}
\def\ep{{\mathbf{\,e}}_p}
\def\eq{{\mathbf{\,e}}_q}
\def\er{{\mathbf{\,e}}_r}
\def\es{{\mathbf{\,e}}_s}

 \def\SS{{\mathbf{S}}}

 \def\0{{\mathbf{0}}}
 
 \newcommand{\GL}{\operatorname{GL}}
\newcommand{\SL}{\operatorname{SL}}
\newcommand{\lcm}{\operatorname{lcm}}
\newcommand{\ord}{\operatorname{ord}}
\newcommand{\Tr}{\operatorname{Tr}}
\newcommand{\Span}{\operatorname{Span}}

\def\({\left(}
\def\){\right)}
\def\l|{\left|}
\def\r|{\right|}
\def\fl#1{\left\lfloor#1\right\rfloor}
\def\rf#1{\left\lceil#1\right\rceil}
\def\sumstar#1{\mathop{\sum\vphantom|^{\!\!*}\,}_{#1}}

\def\mand{\qquad \mbox{and} \qquad}

\def\tblue#1{\begin{color}{blue}{{#1}}\end{color}}

%%%%%%%%%%%%%%%%%%%%%%%%%%%%%%%%%%%%%%%%%%%%%%%%%%%%%%%%
%%%%%%%%%%%%%%%%%%%%%%%%%%%%%%%%%%%%%%%%%%%%%%%%%%%%%%%%
%%%%%%%%%%%%%%%%%%%%%%%%%%%%%%%%%%%%%%%%%%%%%%%%%%%%%%%%
%%%%%%%%%%%%%%%%%%%%%%%%%%%%%%%%%%%%%%%%%%%%%%%%%%%%%%%%

%%%%%%%  END OF STANDARD STUFF %%%%%%%%%

%%%%%%%%%%%%%%%%%%%%%%%%%%%%%%%%%%%%%%%%%%%%%%%%%%%%%%%%
%%%%%%%%%%%%%%%%%%%%%%%%%%%%%%%%%%%%%%%%%%%%%%%%%%%%%%%%
%%%%%%%%%%%%%%%%%%%%%%%%%%%%%%%%%%%%%%%%%%%%%%%%%%%%%%%%
%%%%%%%%%%%%%%%%%%%%%%%%%%%%%%%%%%%%%%%%%%%%%%%%%%%%%%%
%%%%%%%%%%%
%%% Spell

\hyphenation{re-pub-lished}

\mathsurround=1pt

\def\bfdefault{b}

\def \F{{\mathbb F}}
\def \K{{\mathbb K}}
\def \N{{\mathbb N}}
\def \Z{{\mathbb Z}}
\def \P{{\mathbb P}}
\def \Q{{\mathbb Q}}
\def \R{{\mathbb R}}
\def \C{{\mathbb C}}
\def\Fp{\F_p}
\def \fp{\Fp^*}

\def\PER{{\mathcal{PER}}}
\def\per{{\mathrm {per}\,}}
\def\imm{{\mathrm {imm}\,}}

 \def \xbar{\overline x}
 \def \Kbar{\overline K}
  \def \Fbar{\overline \F_p}
   \def \Qbar{\overline \Q}

\title{On some matrix counting problems}

 \author[A. Mohammadi] {Ali Mohammadi}
\address{School of Mathematics and Statistics, University of New South Wales, Sydney NSW 2052, Australia}
\email{ali.mohammadi.np@gmail.com}

\author[A. Ostafe] {Alina Ostafe}
\address{School of Mathematics and Statistics, University of New South Wales, Sydney NSW 2052, Australia}
\email{alina.ostafe@unsw.edu.au}

\author[I. E. Shparlinski] {Igor E. Shparlinski}
\address{School of Mathematics and Statistics, University of New South Wales, Sydney NSW 2052, Australia}
\email{igor.shparlinski@unsw.edu.au}

\begin{abstract}  
We estimate the frequency of singular matrices and of matrices of a given rank whose entries are parametrised by arbitrary polynomials over the integers and modulo a prime $p$. In particular, in the integer case,  we improve a recent bound of V.~Blomer and J.~Li (2022).
\end{abstract}



\maketitle

\tableofcontents

\section{Introduction} 

\subsection{Background and motivation}
Given  an $m \times n $ matrix 
\[
\vf=\(f_{i,j}\(X_{i,j}\)\)_{\substack{1\le i\le m \\ 1 \le j\le n}}
\] 
of univariate polynomials $f_{i,j}\(X_{i,j}\)\in\Z[X_{i,j}]$ we consider the family 
$\cM_\vf$ 
of matrices  with polynomial entries of the form 
\[
\cM_{\vf} = \left\{  \(f_{i,j}\(x_{i,j}\)\)_{1\le i,j\le n}:~ x_{i,j} \in \Z, \  1\le i\le m, \ 1 \le j\le n\right\}.
\]

Furthermore, given an integer $H$,  we consider the set $ \cM_{\vf}(H)$ of $(2H+1)^{mn}$ 
matrices from $\cM_\vf$  with  $x_{i,j} \in [-H,H]$, $1\le i\le m$, $1 \le j\le n$.

Here we are interested in counting matrices from $\cM_{\vf}(H)$ which are of a given rank $r$ and 
denote the number of such matrices by $L_{\vf,r}(H)$. 
Similarly, given a prime $p$ we also consider the number $L_{\vf, r}(H,p)$ of 
 matrices from $\cM_{\vf}(H)$ whose reduction modulo $p$ is  of a given rank $r$ 
 over the finite field $\F_p$ of $p$ elements.

In the case of square matrices, that is, for $m=n$, the questions of counting singular matrices
\begin{equation}\label{eq:Count Sing}
\begin{split}
& N_{\vf}(H) =  \# \{\vX \in \cM_{\vf}(H):~\det \vX=0\}, \\
& N_{\vf}(H, p) =  \# \{\vX \in \cM_{\vf}(H):~\det \vX \equiv  0 \pmod p\}, 
\end{split}
\end{equation} 
are of special interest. 

These   questions are partially  motivated by  recent work  of Blomer and Li~\cite{BlLi}
who have introduced  and estimated  the quantity  $L_{\vf,r}(H)$   in the  special case  
\begin{equation}
\label{eq:Mon}
f_{i,j}(X_{i,j}) = X_{i,j}^d,  \qquad 1\le i\le m, \ 1 \le j\le n,
\end{equation} 
for a fixed integer $d\ge 1$. In fact in~\cite{BlLi} the entries of  $A \in \cM_{\vf}$ belong to a
dyadic interval $x_{i,j} \in [H/2,H]$, but the method can be extended to matrices with 
$x_{i,j} \in [-H,H]$.

\subsection{Previous results}

First we recall that in the 
special case of linear polynomials  $f_{i,j}(X_{i,j})=X_{i,j}$, $1\le i\le m$, $1 \le j\le n$, 
in which case we write  $L_{m,n, r}(H)$ instead of  $L_{\vf, r}(H)$, a result of Katznelson~\cite[Theorem~1]{Katz}, 
used in a very crude form, 
implies that 
\begin{equation}
\label{eq:Katz}
L_{m,n, r}(H) = H^{nr + o(1)}. 
\end{equation}

Furthermore, in the monomial  case~\eqref{eq:Mon}, it is shown in the proof  of~\cite[Lemma~3]{BlLi}  that
\begin{equation}
\label{eq: BL-Bound}
L_{\vf,r}(H)  \le H^{mr + (n-r)(r-1)  + o(1)}
\end{equation} 
for any fixed integers $n \ge m \ge r > 0$. In fact, one can easily see that one can extend 
this to many other choices of polynomials in $\vf$, not necessary monomials as in~\eqref{eq:Mon}. 
Note that for large $n=m$ and $r$ close to $n$ the exponent in~\eqref{eq: BL-Bound} is not too 
far from the exponent in~\eqref{eq:Katz}.

The quantity  $L_{\vf, r}(H,p)$ has not been studied prior this work except for the 
special case of linear polynomials  $f_{i,j}(X_{i,j})=X_{i,j}$, $1\le i\le m$, $1 \le j\le n$, 
in which case we write  $L_{m,n, r}(H,p)$ instead of  $L_{\vf, r}(H,p)$. 
In this case, the asymptotic formula of~\cite[Theorem~9]{AhmShp}   asserts that for $r \le \min\{m,n\}$ for the number $L_{m,n,r}(H, p)$ of such matrices we have
\begin{equation}\label{eq:AhmShp}
\begin{split}
& \left|L_{m,n,r}(H, p) -\frac{1}{p^{(m-r)(n-r)}}  (2H+1)^{mn} \right|\le \\  
 &\qquad\qquad\qquad\qquad \(p^{r(m+n-r)/2} + H^{r(m+n-r)-1}p^{1/2} \)p^{o(1)}. 
\end{split}
\end{equation} 
In fact,~\cite[Theorem~9]{AhmShp} 
 gives a more precise error term with  some logarithmic factors instead of $p^{o(1)}$. 

Furthermore, El-Baz, Lee and Str{\"o}mbergsson~\cite{E-BLS} 
have given matching upper and lower bounds for  $L_{m,n.r}(H,p)$, which for  $n \ge m \ge r > 0$ and  $ 1 \le H \le p/2$  can be written as 
\begin{equation}\label{eq:E-BLS-bound}
\begin{split}
  \max\{H^{mr} & , H^{mn} p^{-(m-r)(n-r)}\} \\
  &\quad  \ll L_{m,n.r}(H,p) \ll  \max\{H^{mr} , H^{mn} p^{-(m-r)(n-r)}\} , 
\end{split}
\end{equation} 
where  the notations $U\ll V$ and $V \gg U$ are both equivalent to the statement $|U|\leq c V$, 
 for some  constant $c> 0$, which throughout this work may depend on the real positive parameters 
 $m$, $n$ and $r$ and also, where obvious, on the polynomials in $\vf$.  
 
\subsection{Description of our results} Here we use a combination of analytic and algebraic arguments to 
 study $L_{\vf,r}(H)$ and $L_{\vf, r}(H,p)$. First, we modify an argument of Blomer and Li~\cite{BlLi}
 and augment it with several new ideas to obtain a substantially stronger version of their bound~\eqref{eq: BL-Bound}, 
 see, for  example,~\eqref{eq:m = n}. 
 This new bound is readily available to be used  in the proof of~\cite[Lemma~3]{BlLi}. It however remains to see whether 
our stronger bound leads to improvements of the  main results of Blomer and Li~\cite{BlLi}. 
 
We also use a similar approach to get an upper bound on  $L_{\vf, r}(H,p)$. 
For $H \ge p^{3/4+ \varepsilon}$ with some fixed $ \varepsilon>0$, 
 using a new result on absolute irreducibility of determinantal varieties, coupled with a result of Fouvry~\cite{Fouv00}, 
 we obtain an asymptotic  formula for  $N_{\vf}(H, p)$.  We also obtain a similar result for vanishing 
 immanants, which are broad generalisations of determinants and permanents. Finally, we pose an open 
 Problem~\ref{prob: AsymForm} concerning an asymptotic formula for $L_{\vf, r}(H,p)$.

\section{Main results} 

\subsection{Results over $\Z$}
We start with the following improvement and generalisation of the bound~\eqref{eq: BL-Bound}.

Here we assume that $d\ge 3$. Note that our approach works for $d=2$ as well, when it becomes 
of the same strength as~\eqref{eq: BL-Bound}, while it still extends it to more general matrices. 

It is convenient to introduce  the parameter  $s_t$,  which  for $t = 3, \ldots, 10$
is given by  Table~\ref{tab:s_t}, while  for $t \ge 11$ we define  $s_t$ as the largest integer $s\le d$ with $s(s+1)\le t+1$.

\begin{table}[h]
	\begin{tabular}{|c|c|c|c|c|c|c|c|c|}
		\hline
		\textbf{$t$} &$3$& $4$& $5$ & $6$ &$7$ & $8$ & $9$ &  $10$ \\ \hline
		$s_t$          & $2$           & $9/4$   &$5/2$  & $11/4$  & $3$   & $3$ &   $3$ &   $3$  \\ \hline
	\end{tabular}
	\vskip 10pt
	\caption{}
\label{tab:s_t}
\end{table}

Next, we define  
\[
\Delta(d,m,n,r)  =   \max_{t =3, \ldots, r}\left\{0,\,  \(t-1\)m -n\(s_t-1\) - r\(t -s_t\)\right\},
\]
where  $s_t$ is the largest integer $s\le d$ with $s(s+1)\le t+1$.
Note that we use the convention that for $r\le 2$ the last term in $ \max_{t =3, \ldots, r}$ is omitted
(or set to zero).

\begin{theorem}\label{thm: rank poly matr fij Z 1}
Let $ n\ge m \ge r \ge 4$. Fix an $m\times n$  matrix  
 $\vf $ of non-constant polynomials $f_{i,j}(X_{i,j})\in \Z[X_{i,j}]$ of degrees $\deg f_{i,j} \ge d\ge 3$, 
  $1\le i\le m$, $1 \le j\le n$. 
 Then,  
 \[
L_{\vf, r}(H) \le H^{m+nr-r+\Delta(d,m,n,r)+o(1)}
\] 
as $H\to  \infty$. 
\end{theorem}

 To see that  Theorem~\ref{thm: rank poly matr fij Z 1}   improves the bound~\eqref{eq: BL-Bound} in a very broad
range of parameters $n\ge m>r\geq 4$
and $d \ge 3$ we observe that $s_t \ge 9/4$ for $t \ge 4$ and so we have
\begin{align*}
\(t-1\)m -n\(s_t-1\) - r\(t -s_t\) & =  n-m +(m-r)t - (n-r) s_t \\
& \le n-m +(m-r)r - 9(n-r)/4  \\
 &=  (m-r)r - m - 5n/4 +9r/4.
\end{align*}
Hence, considering the term corresponding to $t=3$ seprately,
we see that   Theorem~\ref{thm: rank poly matr fij Z 1}, used  in a very crude form,  implies
\begin{equation}\label{eq:Simple form}
\begin{split}
L_{\vf, r}(H) \le H^{m+nr-r+o(1)} &+  H^{3m +n(r-1) - 2r +o(1)} \\
& \qquad \quad + H^{mr + (n-r)(r-5/4)  + o(1)}.
\end{split}
\end{equation} 

\begin{remark}
\label{rem:VL-comparison}
In the setting of~\cite{BlLi}, proportional rows are excluded from consideration.
This means that in this scenario, the bound~\eqref{eq:t in [0,1]} can be dropped in the final
bound in the proof of Theorem~\ref{thm: rank poly matr fij Z 1}  
in Section~\ref{sec:T1 2}. Hence  for such matrices we can replace 
$\Delta(d,m,n,r)$ with just $\max_{t =3, \ldots, r}\left\{\(t-1\)m -n\(s_t-1\) - r\(t -s_t\)\right\}$. 
In particular, for $L_{\vf, r}^\sharp (H)$,  defined  fully analogously  to $L_{\vf, r}(H)$, but for 
matrices with this additional non-proprtionality restriction, instead of~\eqref{eq:Simple form} we have
\[
L_{\vf, r}^\sharp(H) \le  H^{3m +n(r-1) - 2r +o(1)} + H^{mr + (n-r)(r-5/4)  + o(1)},
\]
which is always stronger than the bound~\eqref{eq: BL-Bound}. 
\end{remark}

In the most interesting case $m=n$,  for $r \ge 4$, 
the bound in Theorem~\ref{thm: rank poly matr fij Z 1} becomes 
\begin{equation}\label{eq:m = n}
L_{\vf, r}(H)     \le H^{nr+(n-r)(r-s_r+1)+o(1)}. 
\end{equation} 
%%which improves the bound~\eqref{eq: BL-Bound} for all $n>r\geq 4$ (since $s_r > 2$ in this case).
Note that here we have used that 
$$
\max_{t =3, \ldots, r}(t-s_t) = r-s_r,
$$ 
since $t-s_t$ grows monotonically, which follows from the observation $s_{t+1}-s_t \leq 1 = (t+1)-t$.
In particular, for $N_\vf(H)$, given by~\eqref{eq:Count Sing}, 
we have  
\begin{equation}\label{eq:Sing Matr}
N_\vf(H) \le  H^{n^2-s_{n-1}+o(1)}, 
\end{equation} 
while~\eqref{eq: BL-Bound} gives $N_\vf(H) \le  H^{n^2-2+o(1)}$.  Thus if $d\ge n^{1/2}$ then
we save about $n^{1/2}$ against the trivial bound.

\begin{remark}
\label{rem:Pseudocomparison}
While we compare the exponents in~\eqref{eq: BL-Bound} and~\eqref{eq:m = n} with that in~\eqref{eq:Katz}, 
we do not have any convincing argument to suggest that the rank statistics of  matrices with non-linear polynomials
 has to resemble that 
of matrices with linear polynomials. Note that in the case of equal polynomials $f_{ij}(X) = f(X)$, 
$1\le i\le m$, $1 \le j\le n$, one can easily show that 
\[
L_{\vf,r}(H)  \gg H^{nr}
\]
by first choosing $x_{i,j}$ such that $ \(f_{i,j}\(x_{i,j}\)\)_{1\le i,j\le r}$ is non-singular 
(in $(2H+1) ^{r^2} + O\(H^{r^2-1}\)\gg H^{r^2}$ ways);
choosing the remaining entries in the first $r$ rows arbitrary in  $(2H+1)^{(n-r)r}$ ways, 
and setting 
$x_{h,j} = x_{1,j}$ for all $h =r +1, \ldots, m$ and $j=1, \ldots, n$.
A similar comment also applies to the comparison
between the  bound~\eqref{eq:E-BLS-bound} and our  bound~\eqref{eq:m = n  Fp} below. 
\end{remark}

\begin{remark}
Among other ingredients, our proof of Theorem~\ref{thm: rank poly matr fij Z 1} relies on a result 
of Pila~\cite{Pila}. Under various additional assumptions on the polynomials  
$f_{i,j}$,  $1\le i\le m$, $1 \le j\le n$, one can use stronger bounds such as 
of  Browning and  Heath-Brown~\cite{BrHB1,BrHB2} or Salberger~\cite{Salb}. 
However this does not change the final result as other bounds dominate this part 
of the argument. On the other hand, for $d\ge 3$ some further improvements are possible
as described in Section~\ref{sec:improve}. 
\end{remark}

\begin{remark}
A more general, although quantitatively weaker version of  Theorem~\ref{thm: rank poly matr fij Z 1} 
concerning counting matrices with entries from some rather general convex sets (rather than polynomial images), may be obtained through the use of~\cite[Theorem~3]{BrHaRu} and~\cite[Theorem~27]{Shk} instead of our  Lemmas~\ref{lem:SepVars}, \ref{lem:Small k} and~\ref{lem:Pila}. 
\end{remark}

Now for $a \in \Z$ we denote 
\[
N_{\vf}(H;a) =  \# \{\vX \in \cM_{\vf}(H):~\det \vX=a\}.
\]
Thus $N_{\vf}(H) = N_{\vf}(H;0)$. 

\begin{cor}
\label{cor:det}
Let $ n  \ge 4$. Fix an $n\times n$  matrix  
 $\vf $ of non-constant polynomials $f_{i,j}(X_{i,j})\in \Z[X_{i,j}]$ of degrees $\deg f_{i,j} \ge d\ge 3$, 
  $1\le i, j\le n$. 
 Then,  uniformly over $a \in \Z$ we have 
 \[
N_{\vf}(H;a) \le H^{n^2-s_{n-1}+o(1)}
\] 
as $H\to  \infty$. 
\end{cor}
 
\subsection{Results over $\F_p$}
\label{sec: Res Fp}

We first remark that using bounds of~\cite{Chang} and~\cite{KMS} instead of the bounds in Section~\ref{sec: sol box Z}  
one can derive analogues of  Theorem~\ref{thm: rank poly matr fij Z 1}  for polynomial matrices over a finite field. 
Furthermore, in some ranges of $H$ the bounds from~\cite{Chang, KMS}  can be augmented with bounds of exponential sums with polynomials based on the Vinogradov mean value theorem, see, for example,~\cite[Theorem~5]{Bourg} or, depending on the range of $H$, the classical Weil's bound, see, for example,~\cite[Chapter~6, Theorem~3]{Li}  or~\cite[Theorem~5.38]{LN}. 
For such $H$ this leads to a stronger version of the trivial inequality~\eqref{eq:Triv Jk}. 

To show the ideas and to avoid the unnecessary clutter, we only consider the case of small $H$, 
where the result takes the simplest form. We recall that $L_{\vf, r}(H,p)$  is the number of 
 matrices from $\cM_{\vf}(H)$ whose reduction modulo $p$ is  of a given rank $r$ 
 over the finite field $\F_p$ of $p$ elements.  In fact, in this case, the result is uniform 
 with respect to the polynomials in $\vf$, which can now be assumed to be defined over $\F_p$ 
 rather than over $\Z$ (as we need in Theorems~\ref{thm: sing poly matr fij Fp} 
 and~\ref{thm: imm poly matr fij Fp} below).

\begin{theorem}\label{thm: rank poly matr fij Fp}
Let $ n\ge m \ge r \ge 3$. Fix an $m\times n$  matrix  
 $\vf $ of non-constant polynomials $f_{i,j}(X_{i,j})\in \F_p[X_{i,j}]$,  of degrees $e \ge \deg f_{i,j} \ge 2$, 
   $1\le i\le m$, $1 \le j\le n$. 
 Then,  for 
\[
H \le p^{2/(e(e+1))}
\]
we have 
 \[
L_{\vf, r}(H,p) \le   H^{m+nr-r+\Gamma(m,n,r)+o(1)},
\]
where
\[
\Gamma(m,n,r) = \max\{0, \, m- (n+r)/2, \, m(r-1) - n -r(r-2)\},
\]
as $H\to  \infty$. 
\end{theorem}

For $m=n$, the bound in Theorem~\ref{thm: rank poly matr fij Fp} becomes 
\begin{equation}\label{eq:m = n  Fp}
L_{\vf, r}(H,p) \ll H^{n(r+1) -r+(n-r)(r-2)+o(1)} = H^{nr+(n-r)(r-1)+o(1)}. 
\end{equation} 
As in the case of polynomial matrices over $\Z$ we note that in the corresponding range of $H$, 
for $r$ close to $n$,  the exponent in~\eqref{eq:m = n  Fp} is not too 
far from the exponent in~\eqref{eq:E-BLS-bound}, see however Remark~\ref{rem:Pseudocomparison}.

Next, for $m=n$ we  present an asymptotic formula for the number of singular matrices $N_{\vf}(H, p)$
given by~\eqref{eq:Count Sing}.  
Similarly to the proof of~\eqref{eq:AhmShp} our result is based on a result of Fouvry~\cite{Fouv00} on the distribution of 
rational points on rather general algebraic varieties over prime finite fields.  Our main result is as follows.

\begin{theorem}\label{thm: sing poly matr fij Fp}
Let $n\ge 3$. Fix an $n\times n$  matrix  
 $\vf $ of non-constant polynomials $f_{i,j}(X_{i,j})\in \Z[X_{i,j}]$, $1\le i, j\le n$. Let $p$ be a sufficiently large prime. 
 Then, for a positive integer   $H \le p/2$ we have
\[
N_{\vf}(H, p)= \frac{1}{p} (2H+1)^{n^2}+ O\(p^{(n^2-1)/2 +o(1)} + H^{n^2-2}p^{1/2+o(1)} \),
\]
as $p\to  \infty$.
\end{theorem}

We remark that Theorem~\ref{thm: sing poly matr fij Fp}  is nontrivial for $H\ge p^{3/4+\varepsilon}$ for any fixed $\varepsilon > 0$ and sufficiently large prime $p$. 

Clearly, Theorems~\ref{thm: rank poly matr fij Fp} and~\ref{thm: sing poly matr fij Fp} can be used to derive analogues of Corollary~\ref{cor:det}.

Next, in the special case when $\vf$ consists  of polynomials of the same degree, we obtain a very 
broad generalisation of Theorem~\ref{thm: sing poly matr fij Fp} to   the much wider
class of matrix functions known as {\it immanants\/} which  are expressions of the form
\[ \imm_\chi \vX =
\sum_{\sigma \in \cS_n} \chi(\sigma) \prod_{i=1}^n x_{i,\sigma(i)}, 
\]
where $\vX = \(x_{i,j}\)_{1 \le i,j \le n}$ is an $n\times n$ matrix (over an arbitrary ring) and  $\chi: \cS_n \to \C$ is an
irreducible character of the symmetric group $\cS_n$. In particular, the trivial character
$\chi(\sigma) = 1$ corresponds to the {\it permanent} $\per \vX$, the alternating
character $\chi(\sigma) = \sign \sigma$ corresponds to the
{\it determinant\/} $\det \vX$.

This motivates us to   define the following extension of  $N_{\vf}(H, p)$:
\[
N_{\vf, \chi}(H, p) =  \# \{\vX \in \cM_{\vf}(H):~ \imm_\chi  \vX \equiv  0 \pmod p\}, 
\]
where $\chi$  is an arbitrary character of $\cS_n$.

\begin{theorem}\label{thm: imm poly matr fij Fp}
Let $n\ge 3$. Fix an $n\times n$  matrix  
 $\vf $ of non-constant polynomials $f_{i,j}(X_{i,j})\in \Z[X_{i,j}]$,  $1\le i,  j\le n$, of the same degree $d\ge 1$. 
 Let $p$ be a sufficiently large prime. 
 Then, for any character $\chi$ of $\cS_n$, for a positive integer   $H \le p/2$ we have
\[
N_{\vf, \chi}(H, p)= \frac{1}{p} (2H+1)^{n^2}+ O\(p^{(n^2-1)/2 +o(1)} + H^{n^2-2}p^{1/2+o(1)} \),
\]
as $p\to  \infty$.
\end{theorem} 

We conclude with the following. 

\begin{prob}\label{prob: AsymForm}
Obtain analogues of the asymptotic formulas of 
Theorems~\ref{thm: sing poly matr fij Fp}  and~\ref{thm: imm poly matr fij Fp} for $L_{\vf, r}(H,p)$.
\end{prob} 

The main obstacle towards a resolution of Problem~\ref{prob: AsymForm} is the 
lack of absolute irreducibility result for the corresponding algebraic variety, 
similar to Lemma~\ref{lem:detpolyirr}, which is an interesting question in its 
own rights.

\section{Absolute irreducibility of some  polynomials}

\subsection{Preparations} 
We require the following result  of Tverberg~\cite{Tve}  (see also~\cite[Corollary~2, Section~1.7]{Sch}).

\begin{lemma}\label{lem:diagonalpolyn}
Let $\K$ be an algebraically closed field of characteristic zero, $n\ge 3$ and let $f_{i}\in \K[X_{i}]$, $i=1,\ldots,k$, be non-constant polynomials. Then  the polynomial
\[
  H(X_1, \ldots, X_k)=f_1(X_1)+\cdots+f_k(X_k)
\]
is absolutely irreducible.
\end{lemma}

\begin{remark}
\label{rem:abs irred}
If $\chr \K=p>0$, it is known by the work of Schinzel~\cite[Corollary 3, Section 1.7]{Sch} that the polynomial
\[
  H(X_1, \ldots, X_k)=f_1(X_1)+\cdots+f_k(X_k)
\] 
is absolutely irreducible if and only if at least one polynomial $f_i(X_i)$ is not of the form $h_i(X_i)^p+c h_i(X_i)$, for some $c\in\K$ and some  $h_i\in\K[X_i]$. This condition is indeed needed as the following example shows: for any $c\in\K$ let $f_i(X_i)=X_i^p+cX_i$, $i=1,\ldots,k$, then obviously the polynomial
 \[H(X_1,\ldots,X_k)=(X_1+\ldots+X_k)^p+c(X_1+\ldots+X_k)\] 
 is reducible over $\F_p$.
\end{remark}

For us it will be sufficient to have a result as in Lemma~\ref{lem:diagonalpolyn} when $\chr\K=p>0$ is a sufficiently large prime $p$ and the polynomials are defined over $\Z$, and thus we need  Ostrowski's theorem (see~\cite[Corollary~2B]{Schmidt}), which we state below.

\begin{lemma}\label{lem:Ostrowski}
Let $f(X_1, \ldots, X_k)\in \Z[X_1, \ldots, X_k]$ be an absolutely irreducible polynomial of degree $d$ and let $p$ denote a prime with
\[
p> (4\|f\|)^{M^{2^M}},
\]
where $\|f\|$ denotes the sum of the absolute values of the coefficients of $f$ and $M=\binom{k+d-1}{k}$. Then the reduction of $f$ modulo $p$ is absolutely irreducible over $\F_p$.
\end{lemma}

Therefore, we have the following direct consequence of Lemma~\ref{lem:diagonalpolyn} and Lemma~\ref{lem:Ostrowski}.

\begin{cor}
\label{cor:abs irred p}
Let $k\ge 3$ and  let $f_{i}\in \Z[X_{i}]$, $i=1,\ldots,k$,  be non-constant polynomials. Then, for any sufficiently large prime $p$,  the polynomial
\[
  H(X_1, \ldots, X_k)=f_1(X_1)+\cdots+f_k(X_k)
\]
is absolutely irreducible over $\F_p$.
\end{cor}

\subsection{Absolute irreducibility of  determinant varieties}

We believe that the following result is of independent interest and 
is our main tool in establishing Theorem~\ref{thm: sing poly matr fij Fp}. 

\begin{lemma}\label{lem:detpolyirr} 
Let $n\ge 3$ and let  $f_{i,j}\in \Z[X_{i,j}]$, $i,j=1,\ldots,n$, be non-constant polynomials.
Then the determinant 
    $\det\(f_{i,j}(X_{i,j})\)_{1\leq i,j\leq n}$, viewed as an element of $\Z[X_{1,1}, \ldots, X_{n,n}]$, is absolutely irreducible over $\Q$ and over $\F_p$ for any sufficiently large prime $p$.
\end{lemma}

\begin{proof} 
Let us denote
\[D\(\(X_{i,j}\)_{1\le i,j\le n}\) = \det\(f_{i,j}(X_{i,j})\)_{1\leq i,j\leq n}
\] 
and 
\[
d_{i,j}=\deg  f_{i,j}, \qquad i,j=1,\ldots,n.
\] 

We prove first that $D$ is irreducible over $\C$ 
and then we apply Corollary~\ref{cor:abs irred p} to conclude the absolute irreducibility modulo any sufficiently large prime $p$.

Assume now that $D = fg$ for some $f,g\in \C[X_{1,1},\ldots X_{n,n}]$. We fix a specialisation
\[(\alpha_{i,j}, ~ i=2,\ldots,n, \ j=1,\ldots,n)\in\C^{n(n-1)}
\] of the last $n(n-1)$ indeterminates $X_{2,1},\ldots, X_{n,n}$ such that we obtain
\begin{align*}
 D&(X_{1,1},\ldots, X_{1,n}, \alpha_{2,1}, \ldots, \alpha_{n,n}) \\
 &\qquad\qquad\qquad\qquad\quad= \begin{vmatrix}
    f_{1,1}(X_{1,1}) & f_{1,2}(X_{1,2}) &  \ldots  & f_{1,n}(X_{1,n}) \\
    1 & &   &  \\
    \vdots &  &  I_{n-1} &  \\
    1 &  &   & 
\end{vmatrix}.
\end{align*}
We write $D_*, f_*, g_*$ for the resulting specialised $n$-variable polynomials. To compute $D_*$, let $M_j$ denote the $(n-1)\times(n-1)$ matrix resulting from removing the first row and $j$-th column of the matrix above, so that
\begin{equation}\label{eqn:D*1}
    D_*(X_{1,1},\ldots, X_{1,n}) = \sum_{j=1}^n (-1)^{j+1}\cdot \det M_j \cdot f_{1,j}(X_{1,j}).
\end{equation}
Clearly, $\det M_1=1$. To compute $\det M_j$, for $2\leq j\leq n$, write $K_j$ for the matrix resulting from replacing the $j$-th column of $I_{n-1}$ by $[1,1, \ldots, 1]^t$ and note that $\det K_j = \det I_{n-1}=1$. 

Furthermore, for $2\leq j\leq n$, one gets $K_j$ by swapping columns of $M_j$, $j-2$ consecutive times and so $\det M_j= (-1)^{j-2}\det K_j = (-1)^{j-2}$. Hence, going back to~\eqref{eqn:D*1}, we have
\begin{align*}
D_*(X_{1,1},\ldots, &X_{1,n})\\
& = f_{1,1}(X_{1,1}) + \sum_{j=2}^n (-1)^{2j-1}\cdot f_{1,j}(X_{1,j}) \\
&= f_{1,1}(X_{1,1}) - f_{1,2}(X_{1,2}) - f_{1,3}(X_{1,3})-  \ldots - f_{1,n}(X_{1,n}).
\end{align*}
By Lemma~\ref{lem:diagonalpolyn}, $D_*$ is an absolutely irreducible polynomial, which, together with the assumption $D_* = f_* g_*$ implies 
\begin{align*}
{d_{1,j}}& \leq \max\{\deg_{X_{1,j}} f_* , \deg_{X_{1,j}}g_*\} \\
& \leq \max\{\deg_{X_{1,j}} f, \deg_{X_{1,j}} g\}\leq {d_{1,j}}
\end{align*} for all $1\leq j\leq n$.  That is 
\begin{equation}\label{eqn:degx11f}
    \max\{\deg_{X_{1,j}} f, \deg_{X_{1,j}} g\} = {d_{1,j}}, \qquad  1\leq j\leq n.
    \end{equation}

    Next, we use~\eqref{eqn:degx11f} to show that $D$ is absolutely irreducible. In particular, we use the following two basic observations:
    \begin{enumerate}
    \item[(i)]  
         If $h$ is a monomial appearing in $D$, such that $X_{i,j} \mid h$ for some $1\leq i,j\leq n$, then $X_{i,k}$, $X_{k, j} \nmid h$ for $1\leq k\leq n$. 

 Indeed, one can see this by using the determinant formula
     \[
     D=\sum_{\sigma\in \mathsf{S}_n} 
     (-1)^{\pi(\sigma )} f_{1,\sigma(1)}(X_{1,\sigma(1)})\cdots f_{n,\sigma(n)}(X_{n,\sigma(n)}),
     \] 
     where the sum is over all permutations $\sigma$ of the set $\{1,\ldots,n\}$ and
     $\pi(\sigma)$ is the parity of $\sigma$. 
  
    \item [(ii)]   
  We have   $\deg_{X_{i,j}} f + \deg_{X_{i,j}} g =d_{i,j}$ for $1\leq i,j \leq n$.
\end{enumerate}

By~\eqref{eqn:degx11f}, suppose without loss of generality  that $\deg_{X_{1,1}} f = d_{1,1}$, which by~(ii)
implies $\deg_{X_{1,1}} g = 0$. Then, the indeterminates $X_{1,j}$ do not appear in $g$ for any $2\leq j\leq n$ as otherwise this would contradict~(i).
To see this, writing $f=AX_{1,1}^{d_{1,1}} + B$ for some polynomials $A, B$, with $\deg_{X_{1,1}} B<d_{1,1}$, we conclude that the coefficient of $X_{1,1}^{d_{1,1}}$, in $fg$, is precisely $Ag$. Now, if $\deg_{X_{1,j}} g>0$, 
we have $\deg_{X_{1,j}} Ag>0$, and thus $X_{1,1}$ and $X_{1,j}$ would divide a same monomial in $D$, contradicting~(i).

Finally, suppose $g$ involves some indeterminate $X_{i,j}$. Then since $\deg_{X_{1,j}} g=0$ for all $1\leq j\leq n$, as above, writing $f=AX_{1,j}^{d_{1,j}} + B$, for some polynomials $A, B$, with $\deg_{X_{1,j}} B<d_{1,j}$, we conclude that the coefficient of $X_{1,j}^{d_{1,j}}$, in $fg$, is precisely $Ag$. This shows again that $X_{1,j}$ and $X_{i,j}$ divide a same monomial in $D$, contradicting~(i). Since this applies for any variable $X_{i,j}$, we obtain that $g$ is constant (and hence $g=1$), which concludes the absolute irreducibility over $\Q$.

Applying now Corollary~\ref{cor:abs irred p}, we conclude the proof.
\end{proof}

\begin{remark}
We note that Lemma~\ref{lem:detpolyirr} holds over any algebraically closed field $\K$ and without any condition on the characteristic $p$ if we impose some extra condition on $f_{i,j}$ for some $i,j=1,\ldots,n$, as noted in Remark~\ref{rem:abs irred}. More precisely, one has the following statement for which the proof follows exactly the same, applying Remark~\ref{rem:abs irred} instead of Corollary~\ref{cor:abs irred p}:

Let $\K$ be an algebraically closed field, $n\ge 3$ and $f_{i,j}\in \K[X_{i,j}]$, $i,j=1,\ldots,n$, non-constant polynomials. If $\chr \K=p>0$, assume also that for some $i,j=1,\ldots,n$, the polynomial $f_{i,j}$ is not of the form $h_{i,j}^p+ch_{i,j}$ for some $c\in\K^*$ and $h_{i,j}\in\K[X_{i,j}]$.

Then the determinant 
    $\det\(f_{i,j}(X_{i,j})\)_{1\leq i,j\leq n}$, viewed as an element of $\K[X_{1,1}, \ldots, X_{n,n}]$, is absolutely irreducible.
\end{remark}

\begin{remark}
We note that Lemma~\ref{lem:detpolyirr} does not necessarily hold for $n=2$, since for example for $f_{i,j}=X_{i,j}^2$, $i,j=1,2$, we have
\begin{align*}
\det\(X_{i,j}^2\)_{1\leq i,j\leq 2} &= X_{1,1}^2X_{2,2}^2 - X_{1,2}^2X_{2,1}^2 \\ &= (X_{1,1}X_{2,2} - X_{1,2}X_{2,1})(X_{1,1}X_{2,2} + X_{1,2}X_{2,1}).
\end{align*}
\end{remark}

\begin{remark}
\label{rem:rank}
It is certainly interesting to obtain a version of Lemma~\ref{lem:detpolyirr} for the variety of matrices 
from $\cM_{\vf}$ of a given rank $r  \le n$. For matrices with linear entries, that is, when $f_{i,j}(X_{i,j})=X_{i,j}$, $i,j=1,\ldots,n$,  such results are known. Indeed, this follows from the observation that, the set of such matrices is the epimorphic image of the irreducible variety $\GL_n\times \GL_n$, under the regular mapping $(A, B) \mapsto AMB^{-1}$, for any fixed $n\times n$ matrix $M$ of rank $r$. See, for example,~\cite{Abh} 
or~\cite[Proposition~1.1]{BruVet}. 
\end{remark}

In the case of polynomials of the same degree we have the following broad 
generalisation of Lemma~\ref{lem:detpolyirr} on  absolute irreducibility of arbitrary linear combinations of minors.

\begin{lemma}\label{lem:lcminorsirr} 
Let $d\geq 1$, $n,r\ge 3$, with $r\leq n$ and let 
\[
s= \binom{n}{r}^2.
\] 
Given non-constant polynomials $f_{i,j}\in \Z[X]$, $i,j=1, \ldots, n$ of the same degree, 
 write $D_h$, $1\leq h\leq s$, for the $r\times r$ minors  of the matrix $\vf = \(f_{i,j}(X_{i,j})\)_{1\leq i,j \leq n}$.
Then any non-trivial linear combination 
\begin{equation}\label{eqn:linComb}
    \sum_{1\leq  h \leq s} c_h D_h,\qquad \(c_1, \ldots, c_s\)\in \K^s \setminus \{(0, \ldots, 0)\}, 
\end{equation}
is absolutely irreducible over the field $K$ where 
\[
\K = \Q \qquad \text{or} \qquad \K= \F_p
\] 
for any sufficiently large prime $p$.
\end{lemma}

\begin{proof}
Let  $\deg f_{i,j} = d\geq 1$, $1 \le i,j \le n$.
We begin by noting that since linear combinations of corresponding minors of the matrix $\(X_{i,j}^d\)_{1\leq i,j \leq n}$ appear as the homogeneous part of top degree of~\eqref{eqn:linComb}, their absolute irreducibility imply that of~\eqref{eqn:linComb}. Thus it suffices to consider only the case $f_{i,j}= X_{i,j}^d$.

The proof is by induction on the number of nonzero terms appearing in~\eqref{eqn:linComb}, denoted by $t$, noting that the case $t=1$ has been settled by Lemma~\ref{lem:detpolyirr} (here is why we need $p$ to be 
sufficiently large).
Renumbering, we can assume that $c_1, \ldots, c_t \ne 0$ and 
write $\cS = \{  D_\nu:~1\leq \nu \leq t\}$ for the minors appearing in~\eqref{eqn:linComb}. Let $t\geq 2$ and suppose the desired result holds for $t-1$. 

Note that one may choose a row or column of the matrix $\cM$ giving a non-trivial partition $\cS = \cS_1 \sqcup \cS_2$, according to whether a given minor belonging to $\cS$ takes entries from that row/column or not. This follows from the fact that two distinct minors may not match in all rows and columns.

Write
\[
R = \sum_{1\leq h \leq t} c_h D_h  =R_1  + R_2,
\]
such that $R_\nu$ corresponds to the sum of terms appearing in $\cS_\nu$ for $\nu=1,2$.

Suppose, without loss of generality, that the minors appearing in $R_1$ and $R_2$ differ in the row $i_0$ and set $X_{i_0, j} = 1$, for $1\leq j\leq n$. Write
\[
R^* = R_1 + R_2^*
\]
for this specialisation and note that $R_1$ and $R_2^*$ are homogeneous polynomials of degrees $rd$ and $(r-1)d$, respectively.

Let 
\[
H(X_{1,1}, \ldots, X_{n,n}, Z) = Z^{nd}\cdot R\(\frac{X_{1,1}}{Z}, \ldots, \frac{X_{n,n}}{Z}\).
\] 
That is, $H$ represents the homogenised form of $R$. Note that 
\[
H = R_1 + R_2^* \cdot Z^d.
\]

By the induction hypothesis, $R_1$ is irreducible over the algebraic closure $\overline \K$ of $\K$ (for sufficiently large $p$, if $\K=\F_p)$, and given that $\deg(R_1)>\deg(R_2^*)$, clearly we have $R_1 \nmid R_2^*$. Thus irreducibility of $H$ follows by an application of Eisenstein's criterion. In turn, this implies the irreducibility of $R^*$, as a polynomial is irreducible over  $\overline \K$ if and only if its homogenised form is irreducible (see~\cite[Exercise~9, p.~392]{CoxLitShe}).

Now, suppose that $R = fg$ for some $f,g\in   \overline \K [X_{1,1},\ldots X_{n,n}]$ and write $f^*$ and $g^*$ for the polynomials resulting from setting $X_{i_0, j} = 1$, for $1\leq j\leq n$. Thus $R^* = f^*g^*$. 

Similarly to the arguments of Lemma~\ref{lem:detpolyirr}, irreducibility of $R^*$, implies
\begin{align*}
{d} \leq \max\{\deg_{X_{i,j}} f^* , \deg_{X_{i,j}}g^*\} 
& \leq \max\{\deg_{X_{i,j}} f, \deg_{X_{i,j}} g\}\leq {d},\\   1\leq i,j\leq n, & \quad i\not=i_0.
\end{align*} 
That is,
\begin{equation}\label{eqn:degx11f2}
 \max\{\deg_{X_{i,j}} f, \deg_{X_{i,j}} g\} = d \qquad  1\leq i,j\leq n, \quad i\not=i_0.
 \end{equation}

 Furthermore, clearly both observations~(i) and~(ii), in the proof, of Lemma~\ref{lem:detpolyirr} hold for the polynomial $R$. Thus the remainder of the proof is essentially a repetition of the arguments of Lemma~\ref{lem:detpolyirr} and is streamlined.

 By~\eqref{eqn:degx11f2}, without loss of generality, let $\deg_{X_{1,1}} f = d$. Then, by (ii) in the proof of Lemma~\ref{lem:detpolyirr},
we have $\deg_{X_{1,1}} g = 0$ and thus by observation (i) in the proof of Lemma~\ref{lem:detpolyirr}, $g$ cannot involve the indeterminates $X_{1,j}$ for any $2\leq j\leq n$.

Furthermore if $g$ involves some indeterminate $X_{i,j}$, since $\deg_{X_{1,j}} g=0$ for all $1\leq j\leq n$, this again contradicts (i) in the proof of Lemma~\ref{lem:detpolyirr}. Consequently, we obtain that $g$ is constant (and hence $g=1$), concluding the absolute irreducibility over $ \overline \K$.
\end{proof}

\section{Point counting on some hypersurfaces}

\subsection{Solutions to polynomial   equations in a box}  
\label{sec: sol box Z}

For a polynomial  $f(X)\in \Z[X]$   and a real $\kappa> 0$ 
we denote 
\[
I_{\kappa}(f,H) =  \int_0^1 \left| \sum_{x=-H}^H \e\(\alpha f(x)\) \right|^{\kappa} \, d\alpha , 
\]
where $\e(z) = \exp(2\pi i z)$. 

We now recall the following result of Wooley~\cite[Corollary~14.2]{Wool}.

\begin{lemma}\label{lem:Wool}
Let $f(X)\in \Z[X]$ be a fixed polynomial of   degree $d\ge 1$. Then for each integer $s$ with $1 \le s \le d$, we
have 
\[
I_{s(s+1)}(f,H) \le H^{s^2 + o(1)},  \qquad  H \to \infty. 
\]
\end{lemma}

Furthermore, for some parameters the following result of Hua~\cite{Hua} gives better estimates.

\begin{lemma}\label{lem:Hua}
Let $f(X)\in \Z[X]$ be a fixed polynomial of   degree $d\ge 1$. Then for each integer $s$ with $1 \le s \le d$, we
have 
\[
I_{2^s}(f,H) \le H^{2^s - s + o(1)},  \qquad  H \to \infty. 
\]
\end{lemma}

We are now able to summarise  our bound on $I_{k}(f,H)$  for integers $k \le 11$. 
It is convenient to define
\begin{equation}\label{eq: sigma_k}
\sigma_{k} = s_{k-1}, \qquad k = 4, \ldots, 11,
\end{equation}
which corresponds to the values of $s_t$ in Table~\ref{tab:s_t}.

\begin{lemma}\label{lem:Ik_small k}
Let $f(X)\in \Z[X]$ be a fixed polynomial of   degree $d\ge 3$. Then for 
$k = 4, \ldots, 11$ we have 
\[
I_{k} (f,H) \le H^{k -\sigma_k + o(1)},   \qquad  H \to \infty. 
\]
\end{lemma} 

\begin{proof} Clearly,  for $k=4$ and $k=8$,  the result follows from Lemma~\ref{lem:Hua} (recalling that $d\ge 3$), 
 taken with $s=2$ and $s =3$, respectively.
 For $k =9,10, 11$ we simply use the trivial bound
 \[
I_{k}(f,H)  \ll H^{k-8}  I_{8}(f,H) .
\]

Next we consider $k=5$ and note that by the H{\"o}lder inequality 
\[
I_{5}(f,H) \le \(I_{4}(f,H)\)^{3/4}  \(I_{8}(f,H) \)^{1/4} \le  H^{11/4 + o(1)} . 
\]
Similarly 
\[
I_{6}(f,H) \le \(I_{4}(f,H)\)^{1/2}  \(I_{8}(f,H) \)^{1/2} \le  H^{7/2 + o(1)} , 
\]
and 
\[
I_{7}(f,H) \le \(I_{4}(f,H)\)^{1/4}  \(I_{8}(f,H) \)^{3/4} \le  H^{17/4 + o(1)}, 
\]
which concludes the proof. 
\end{proof}

It is easy to see that by the orthogonality of exponential functions, $I_{2k}(f,H)$ 
is the number of solutions to  the Diophantine equation
\[
\sum_{i=1}^k f(x_i) =\sum_{i=1}^k f(y_i) , \qquad  -H  \le x_i, y_i \le H, \ i =1, \ldots, k.
\]
We use a generalised form of this observation, together with Lemma~\ref{lem:Wool}
to estimate the number of solutions of a more general equation. 

Given $k$ polynomials  $f_i(X)\in \Z[X]$ and an  integer vector $\va=(a_1, \ldots, a_k)\in \Z^k$, 
for an integer $H\ge 1$ we denote by  $T_\va(f_1, \ldots, f_k; H)$  the  number 
of solutions to  the Diophantine equation
\[
\sum_{i=1}^k a_i f_i(x_i) =0  , \qquad -H \le x_i \le H, \ i =1, \ldots, k. 
\]

\begin{lemma}\label{lem:SepVars}
Let $f_i(X)\in \Z[X]$, $i=1,\ldots,k$, be  $k$ fixed polynomials of  degrees $\deg f_i \ge d\ge 2$,  
 and let $\va=(a_1, \ldots, a_k)\in \Z^k$ be an arbitrary integer vector with nonzero 
 components $a_i\ne 0$,  $i =1, \ldots, k$. 
Then for each positive  integer $s$ such that  $s \le d$ and $s(s+1) \le k$   we have 
\[T_\va(f_1, \ldots, f_k; H) \le   H^{k-s+o(1)},  \qquad  H \to \infty. 
\]
 \end{lemma}
 
 \begin{proof}
 As in the above, by the  orthogonality of exponential functions, we write
\begin{align*}
 T_\va(f_1, \ldots, f_k; H) &  = 
  \int_0^1  \prod_{i=1}^k  \sum_{x_i=-H}^H \e\(\alpha a_i f_i(x_i)\)  \, d\alpha\\
& \le   \int_0^1  \prod_{i=1}^k \left| \sum_{x_i=-H}^H \e\(\alpha a_i f_i(x_i)\) \right| \, d\alpha .
  \end {align*}
 Hence, by the H{\"o}lder inequality 
\begin{equation}\label{eq: T and I}
  T_\va(f_1, \ldots, f_k; H)   \le   \prod_{i=1}^k \(\int_0^1  \left| \sum_{x_i=-H}^H \e\(\alpha a_i f_i(x_i)\) \right|^k
 \, d\alpha\)^{1/k}.
\end{equation}
Since the function $\e(z)$ is periodic with period $1$,  and $a_i \ne 0$, for each $i =1, \ldots, k$,
 we have 
 \begin{align*}
 \int_0^1  \left| \sum_{x_i=-H}^H \e\(\alpha a_i f_i(x_i)\) \right|^k
 \, d\alpha & = \frac{1}{a_i} \int_0^1  \left| \sum_{x_i=-H}^H \e\(\alpha a_i f_i(x_i)\) \right|^k
 \, d(\alpha a_i)\\
  & = \frac{1}{a_i} \int_0^{a_i}   \left| \sum_{x_i=-H}^H \e\(\beta f_i(x_i)\) \right|^k
 \, d\beta\\
  & =  \int_0^1  \left| \sum_{x_i=-H}^H \e\(\alpha   f_i(x_i)\) \right|^k
 \, d \alpha\\
 &= I_k(f_i,H) 
 \end {align*} 
 (we remark that the above calculation holds for both positive and negative 
 values of $a_i$). 
 Hence we derive from~\eqref{eq: T and I} that 
\begin{equation}
\label{eq:Tk vs Ik}
  T_\va(f_1, \ldots, f_k; H)  \le   \prod_{i=1}^k  I_{k}(f_i,H)^{1/k}.
\end{equation} 
For  $k \ge s(s+1)$ we can use   the trivial bound  
\begin{equation}
\label{eq: Ik - triv}
I_{k}(f_i, H) \le  H^{k - s(s+1)} I_{s(s+1)}(f_i, H)
\end{equation} 
and since $s \le d \le \deg f_i$, $i =1, \ldots, k$, Lemma~\ref{lem:Wool} applies, and after simple calculations implies the 
desired result.   \end{proof} 

 \begin{remark}
\label{rem:Better I_k}
Clearly instead of~\eqref{eq: Ik - triv}, assuming that \
\[s(s+1) \le k < (s+1)(s+2),\]
 one can use the H{\"o}lder inequality as in the proof of Lemma~\ref{lem:Ik_small k},  and estimate 
\[
I_{k}(f, H) \le \(I_{s(s+1)}\(f, H\)\)^{1/\alpha} \(I_{(s+1)(s+2)}\(f, H\)\)^{1-1/\alpha},
\]
with
\[
\alpha = \frac{2 (s+1)}{\(s+1\)\(s+2\) - k}.
\]
However, for large $k$ this leads to somewhat cluttered formulas, 
while providing only marginal improvements. 
\end{remark}

For the small values $k=4,\ldots, 11$, 
we obtain  better bounds on $T_\va(f_1, \ldots, f_k; H)$ than
in Lemma~\ref{lem:SepVars}, namely, using   Lemma~\ref{lem:Ik_small k}
in the inequality~\eqref{eq:Tk vs Ik},  we obtain the following bound.

\begin{lemma}\label{lem:Small k}
Let $f_i(X)\in \Z[X]$, $i=1,\ldots,k$, be  $k$ fixed polynomials of  degrees $\deg f_i \ge d\ge 3$,  
 and let $\va=(a_1, \ldots, a_k)\in \Z^k$ be an arbitrary integer vector with nonzero 
 components $a_i\ne 0$,  $i =1, \ldots, k$. 
Then for  $k=4,\ldots, 11$ we have 
\[
T_\va(f_1, \ldots, f_k; H) \le H^{k-\sigma_k+o(1)}, \qquad  H \to \infty, 
\]
where $\sigma_k$ is given by~\eqref{eq: sigma_k}. 
 \end{lemma}
 
Next, we can  also use a general bound of Pila~\cite[Theorem~A]{Pila},
 which applies  by Lemma~\ref{lem:diagonalpolyn} (we only use it for $k=3$ but 
 present it in full generality).

\begin{lemma}\label{lem:Pila} Let $f_i(X)\in \Z[X]$, $i=1,\ldots,k$, be  $k$ fixed polynomials of degrees $\deg f_i \ge d\ge 1$,  
 and let $\va=(a_1, \ldots, a_k)\in \Z^k$ be an arbitrary integer vector with nonzero 
 components $a_i\ne 0$,  $i =1, \ldots, k$. 
 Then
 \[
T_\va(f_1, \ldots, f_k; H) \le   H^{k-2+1/d+o(1)}, \qquad  H \to \infty. 
\]
 \end{lemma}

It is important to observe that for $d \ge 3$  all bounds of this section 
are of the form
\[
T_\va(f_1, \ldots, f_k; H)  \le H^{k - \rho+ o(1)}, 
\]
which is uniform with respect to the coefficients $a_1, \ldots, a_k$, and 
\begin{itemize}
\item for $k \ge 12$, we use  Lemma~\ref{lem:SepVars} (with some $s \ge 3$ but with $s(s+1) \le k$),
which  allows us to take  $\rho \ge 3$;
\item for $11\ge k \ge 4$, we use  Lemma~\ref{lem:Small k}  which  allows us to take  $\rho =\sigma_k$;
\item for $k=3$,  we use  Lemma~\ref{lem:Pila},  which  allows us to take $\rho= 2-1/d$.
\end{itemize}

\subsection{Solutions to polynomial congruences in a box}  
\label{sec: sol box Fp}

As we have mentioned in Section~\ref{sec: Res Fp}, 
 some analogues of the results from Section~\ref{sec: sol box Z} 
can be extracted from~\cite{Chang,KMS}. This leads to a large 
variety of results. We concentrate on the simplest (at least in typographic 
sense) case of small boxes. 
Since the argument is a discrete version of that of Section~\ref{sec: sol box Z}
we are rather brief in our exposition here.

Note that below we freely switch between the language of congruences and 
the language of finite fields. 

For a polynomial  $f(X)\in \F_p[X]$, we denote 
\[
J_k(f,H,p) =  \frac{1}{p} \sum_{\alpha \in \F_p} \left| \sum_{x=-H}^H \ep\(\alpha f(x)\) \right|^k, 
\]
where $\ep(z) = \exp(2\pi i z/p)$.

The following bound is a special case of~\cite[Theorem~1.3]{KMS}.

\begin{lemma}\label{lem:KMS}
Let $f(X)\in \Z[X]$ be a fixed polynomial of  degree $d\ge 2$.  Then,  for 
\[
H \le p^{2/(d(d+1))}
\] 
we have 
\[
J_4(f,H,p) \le H^{2+ o(1)}, \qquad  H \to \infty.
\]
\end{lemma}

Given $k$ polynomials  $f_i(X)\in \F_p[X]$ and a  vector $\va=(a_1, \ldots, a_k)\in \F_p^k$, 
for an integer $H\ge 1$ we denote by  $T_\va(f_1, \ldots, f_k; H,p)$  the  number 
of solutions to  the congruence
\[
\sum_{i=1}^k a_i f_i(x_i) \equiv 0 \pmod  p , \qquad -H \le x_i \le H, \ i =1, \ldots, k. 
\]

\begin{lemma}\label{lem:Energy Fp}
Let $f_i(X)\in \F_p[X]$, $i=1,\ldots,k$, be  $k$ fixed polynomials of  degrees $\deg f_i \ge d\ge 2$,  
 and let $\va=(a_1, \ldots, a_k)\in \F_p^k$ be an arbitrary  vector with nonzero 
 components $a_i\ne 0$,  $i =1, \ldots, k$. 
Then for $H \le p^{2/(d(d+1))}$,  
we have 
\[
T_\va(f_1, \ldots, f_k; H,p ) \le  
\begin{cases} H^{3/2+o(1)} & \text{if}\ k = 3,\\
H^{k-2+o(1)} & \text{if}\ k \ge 4, 
\end{cases}  \qquad  H \to \infty. 
\]
 \end{lemma}
 
 \begin{proof} 
 As in the proof of Lemma~\ref{lem:SepVars}, using the orthogonality of exponential sums and the H\" older inequality, one obtains the analogue of~\eqref{eq: T and I}, that is,
\begin{equation}
\label{eq:T and J}
 T_\va(f_1, \ldots, f_k; H,p )\le \prod_{i=1}^k J_k(f_i,H,p)^{1/k}.
\end{equation}

For each $i=1,\ldots,k$, using the  H{\"o}lder inequality for $k = 3$  gives us  $J_3(f_i,H,p) \le J_4(f_i,H,p)^{3/4}$
 while for $k \ge 4$, we have the trivial bound
\begin{equation}
\label{eq:Triv Jk}
J_k(f_i,H,p) \le J_4(f_i,H,p) H^{k-4+ o(1)}. 
\end{equation} 
We now see that Lemma~\ref{lem:KMS} implies
 \[
J_k(f_i,H,p) \le 
\begin{cases} H^{3/2+o(1)} & \text{if}\ k = 3,\\
H^{k-2+o(1)} & \text{if}\ k \ge 4, 
\end{cases} \qquad  H \to \infty, 
\]
provided $H \le p^{2/(d(d+1))}$. 

Plugging these in~\eqref{eq:T and J}, we derive the 
desired bound.  
 \end{proof}

We now present another important technical tool, given by 
the work of Fouvry~\cite{Fouv00}.  

For a  
polynomial  in $k \ge 2$ variables
\[
F(X_1, \ldots, X_k) \in \Z[X_1, \ldots, X_k]
\] 
and an integer  $H\ge 1$, 
we denote by $T_F(H, p)$ the number of solutions to the congruence
\[
F(x_1, \ldots, x_k) \equiv 0 \pmod p, \qquad (x_1, \ldots, x_k) \in [-H, H]^k,
\]
modulo a prime $p$.  

We also define $T_F(p) = T_F((p-1)/2, p) $, that  is, $T_F(p)$ is 
the number of solutions to  the above congruence with unrestricted variables from  $\F_p$. 
Then, the main result of~\cite[Theorem]{Fouv00}, in the case of one polynomial takes the following form.

For a polynomial $F\in\Z[X_1,\ldots,X_k]$, we denote by $\cZ(F)$ the algebraic subset of $\C^k$ of all zeros of $F$ in $\C^k$.

\begin{lemma}
\label{lem:box} Let 
\[
F(X_1, \ldots, X_k) \in \Z[X_1, \ldots, X_k]
\] 
be an irreducible over $\C$ 
polynomial  in $k \ge 2$ variables, such that the hypersurface $\cZ(F)$ is not contained in any hyperplane of $\C^k$. 
Then for any  prime $p$,  
we have 
\[
T_F(H, p) = \(\frac{H}{p}\)^k T_F(p)  + O\(p^{(k-1)/2 + o(1)} + H^{k-2}  p^{1/2 + o(1)}\)
\]
as $p \to \infty$. 
\end{lemma}

\section{Proofs of main results} 

\subsection{Proof of Theorem~\ref{thm: rank poly matr fij Z 1}}%% and~\ref{thm: rank poly matr fij Z 2}}
\label{sec:T1 2}
As in~\cite{BlLi} we observe that  it is enough to estimate the number $L_{\vf, r}^*(H)$ 
of matrices $\vX \in \cM_{\vf}(H)$ which are of rank $r$ and such that the top left $r \times r$ 
minor $\vX_r = \(f_{i,j}\(x_{i,j}\)\)_{1\le i,j\le r} $ is non-singular.  We now fix the values of 
such $x_{i,j}$, $  i,j=1, \ldots, r$, in 
\begin{equation}
\label{eq:contrib Xr}
\fA \ll H^{r^2} 
\end{equation}
ways. 

We now observe that  for every integer  $h$, $r < h \le m$, once the minor $\vX_r$ is
fixed, the $h$-th   row of every matrix  $U$ which is counted by  $L_{\vf, r}^*(H)$ is a unique
linear combination of the first $r$ rows with  coefficients $\(\rho_{1}(h) , \ldots , \rho_{r}(h)\) \in \Q^r$.

We say that the $(m-r)\times r$ matrix  
\[
\vY_r =\(f_{h,j}\(x_{h,j}\)\)_{\substack{r+1\le h\le m \\ 1 \le j\le r}}
\] 
(which is directly under $\vX_r$ in $\vX$) is of type $t\ge 0$ if $t$ is the largest number of  non-zeros among 
the coefficients $\(\rho_{1}(h) , \ldots , \rho_{r}(h)\) \in \Q^r$ taken over all 
$h = r+1, \ldots, m$. 

In particular type $t = 0$ corresponds to the zero matrix $\vY_r$. 

Clearly if the  $h$-th   row is of type $t$, that is, $t$ of the coefficients in $\(\rho_{1}(h) , \ldots , \rho_{r}(h)\) \in \Q^r$
 are non-zero, say  $\rho_{1}(h) , \ldots , \rho_{t}(h) \ne 0$, and thus we can choose a $t \times t$ non-singular  
sub-matrix  of the matrix  $\(f_{i,j}\(x_{i,j}\)\)_{\substack{1\le i\le t \\ 1 \le j\le r}}$ . 
Again, without loss of generality we can assume 
that this is 
\[ 
\vX_t =  \(f_{i,j}\(x_{i,j}\)\)_{1\le i, j\le t}.
\] 
This means that each of $O(H^t)$ possible choices 
of $X_{h,1},\ldots,X_{h,t}$  defines the coefficients  
\[
\( \rho_{1}(h) , \ldots , \rho_{r}(h)\) = \( \rho_{1}(h) , \ldots , \rho_{t}(h), 0, \ldots, 0\)
\]
and hence the rest of the values $f_{h,j}(x_{h,j})$, $j = t+1, \ldots, r$. 
Note that this bound is monotonically increasing with $t$ and thus applies to every row
of  matrices $\vY_r$ of type $t$. 

Therefore, for each fixed $\vX_r$, there are 
\begin{equation}
\label{eq:contrib Yr-t}
\fB_t \ll H^{t(m-r)}
\end{equation}
matrices $\vY_r$ of type $t$ (note that this is also true for $t = 0$).

Let now a matrix $\vX_r$ and a matrix $\vY_r$ of type $t$ be  both fixed. 
Hence there is an $h$, $r+1\le h \le m$, such that the $h$-th row can be written as a linear combination of the top $r$ rows.  
 As before, without loss of generality we can assume that the  vector
$ \( \rho_{1}(h) , \ldots , \rho_{r}(h)\) $ contains exactly $t$ non-zero components, 
which  for each $j = r+1, \ldots, n$ leads to an equation 
\begin{equation}
\label{eq:lin rel}
\rho_{1}(h) f_{1,j}\(x_{1,j}\)  + \ldots  + \rho_{r}(h)   f_{r,j}\(x_{r,j}\) = f_{h,j}\(x_{h,j} \)
\end{equation}
with exactly $t+1$ non-zero coefficients. 

After this, all other elements $f_{i,j}\(x_{i,j}\)$, $i \in \{r+1, \ldots, m \} \setminus \{h\}$ and $j=r+1,\ldots,n$, are uniquely  defined by an analogue of the 
relation~\eqref{eq:lin rel},  since now for every $i \in \{r+1, \ldots, m \} \setminus \{h\}$, the left hand side is fixed.

Let $\fC_{t}$, $j = r+1, \ldots, n$, be the largest (taken over all choices of $h\in  \{r+1, \ldots, m \}$ and $j\in  \{r+1, \ldots, n \}$) number of solutions to~\eqref{eq:lin rel} 
in variables $\(x_{1,j},  \ldots , x_{r,j}, x_{h,j}\) \in [-H, H]^{r+1}$.

 Then we can summarise the 
above discussion as the bound 
\begin{equation}
\label{eq: L vs ABC}
L_{\vf, r}^*(H)   \ll \fA   \sum_{t =0}^r \fB_t \fC_t^{n-r}.
\end{equation} 

Under the condition of Theorem~\ref{thm: rank poly matr fij Z 1}, to estimate   $\fC_t$, we apply:
\begin{itemize}
\item the trivial bound $ \fC_t \ll H^r$  if $t \in\{0,1\}$;  
\item  the bound $ \fC_t \ll H^{r-1+1/d+o(1)}$ of Lemma~\ref{lem:Pila}  if $  t =2$; 
 \item  the bound $ \fC_t \ll H^{r+1-s_t+o(1)}$,  which combines  Lemma~\ref{lem:Small k}  if $ 3 \le  t  \le 10$ 
 and Lemma~\ref{lem:SepVars} (used with $s=s_t$  if $t \ge 11$, 
 where,  as before, $s_t\ge 3$ is the largest integer $s\le d$ with $s(s+1)\le t+1$. 
\end{itemize}

Hence,  combining the above bounds with~\eqref{eq:contrib Xr} and~\eqref{eq:contrib Yr-t}
and substituting in~\eqref{eq: L vs ABC} 
we obtain the following contributions to $L_{\vf, r}^*(H) $ for $t =0, \ldots, r$.

For $t \in\{0,1\}$ the total contribution $\fL_{0,1}$ satisfies
\begin{equation}\label{eq:t in [0,1]}
\begin{split}
\fL_{0,1} &\ll H^{r^2} \(\( H^r\)^{n-r}+H^{m-r}\( H^r\)^{n-r}\) \\
&  \ll H^{r^2} H^{m-r}\( H^r\)^{n-r}  = H^{m +nr - r} . 
\end{split}
\end{equation} 

For $t = 2$ the total contribution $\fL_{2}$ satisfies 
\begin{equation}\label{eq:t = 2}
\begin{split}
\fL_{2} &\le H^{r^2+o(1) }H^{2(m-r)} \( H^{r-1+1/d}\)^{n-r}  \\
& = H^{2m +n(r-1+1/d) - r (1+1/d)+o(1)} . 
\end{split}
\end{equation} 

Finally, for  $ t \ge 3$, the total contribution $\fL_{\ge 3}$ satisfies 
\begin{equation}\label{eq:t ge 3}
\begin{split}
\fL_{\ge 3}&\le H^{r^2+o(1) }\sum_{t =3}^r   H^{t(m-r)} \( H^{r+1-s_t}\)^{n-r}  \\
& =\max_{t = 3, \ldots, r} H^{tm +n(r+1-s_t) - r(t +1-s_t) +o(1)} . 
\end{split}
\end{equation} 

Substituting the bounds~\eqref{eq:t in [0,1]},  \eqref{eq:t = 2}, 
and~\eqref{eq:t ge 3} in the inequality 
\[
L_{\vf, r}(H) \ll L_{\vf, r}^*(H) \le \fL_{0,1}  + \fL_{2}   
+ \fL_{\ge 3} , 
\] 
which implies that 
 \[
L_{\vf, r}(H) \le H^{m+nr-r+\widetilde \Delta(d,m,n,r)+o(1)}
\] 
with
\begin{align*}
\widetilde \Delta(d,m,n,r) & = \max
\bigl\{0, \, m -n + (n-r)/d,\\  
& \qquad  \qquad  \qquad   \max_{t =3, \ldots, r}\{ (t-1)m -n(s_t-1) - r(t -s_t)\}\bigr\}.
\end{align*}

It remains to show that the term $ m -n + (n-r)/d$ in $\widetilde \Delta(d,m,n,r)$ never dominates
and thus can be dropped leading to 
\begin{equation}\label{eq:Deltas}
\widetilde \Delta(d,m,n,r) =  \Delta(d,m,n,r)
\end{equation}  
for $d \ge 3$ and $r \ge 4$.

Since $m -n + (n-r)/d \le m -n + (n-r)/3$,  it is sufficient to consider the case of $d=3$. 
For this we notice that  we can clearly assume that $m -n + (n-r)/3 \ge 0$ as otherwise~\eqref{eq:Deltas}
is trivial. 
Thus
\begin{equation}\label{eq:large m}
m  \ge  n -  (n-r)/3 = (2n+r)/3.
\end{equation} 
It is  sufficient to show that the term $3m -5n/4  - 7r/4$,  which corresponding to $t = 4$ in $\widetilde \Delta(d,m,n,r)$, 
satisfies 
\begin{equation}\label{eqT2 vs T3}
3m -5n/4  - 7r/4 \ge   m -n + (n-r)/3, 
\end{equation} 
which is equivalent to 
\[
2m \ge n/4 + 7r/4 +  (n-r)/3 = 7n/12 + 17r/12, 
\]
which in turn follows from~\eqref{eq:large m} since for  $n > r$ we have
\[
2(2n+r)/3 \ge 7n/12 + 17r/12.
\]
This implies~\eqref{eqT2 vs T3}, which means that~\eqref{eq:Deltas} 
holds,  and this concludes the proof.

\subsection{Proof of Corollary~\ref{cor:det}}
We can assume that $a \ne 0$ as otherwise the result follows from~\eqref{eq:Sing Matr}. 
We now write 
\[
\det \(f_{i,j}\(X_{i,j}\)\)_{1\le i,  j\le n}
= \sum_{h=1} (-1)^{h-1}f_{1,h}\(X_{1,h}\) D_h, 
\]
where $D_j$ are the determinants of the minors supported on the variables 
$X_{i,j}$, $2 \le i \le n$, $1\le j \le n$, $j \ne h$.

We see from Lemmas~\ref{lem:SepVars} and~\ref{lem:Small k}
that the solutions with $D_h \ne 0$, $h=1, \ldots,  n$, contribute at most 
\[
H^{n(n-1)} H^{n-s_{n-1} + o(1)}  = H^{n^2-s_{n-1}+o(1)}.
\]

On the other hand, by~\eqref{eq:Sing Matr}, using that $a \ne0$, we see that 
the contribution from other solutions can be 
estimated as 
\[
H^{(n-1)^2-s_{n-2}+o(1)} H^{n-1} H^{n-1}  =  H^{n^2-s_{n-2} - 1+o(1)}.
\]
Since $s_{n-1}\le s_{n-2}+1$, the result follows.

\subsection{Proof of Theorem~\ref{thm: rank poly matr fij Fp}} 
We mimic the proof of Theorem~\ref{thm: rank poly matr fij Z 1}, but we use the bounds of Lemma~\ref {lem:Energy Fp} instead of the bounds 
of Section~\ref{sec: sol box Z}.

In particular,  we introduce the same quantities $\fA$, $\fB_t$ and $\cC_t$, but defined for matrices 
over $\F_p$. Then, 
 under the condition of Theorem~\ref{thm: rank poly matr fij Fp}, to estimate   $\fC_t$, we apply:
\begin{itemize}
\item the trivial bound $ \fC_t \ll H^r$  if $t \in\{0,1\}$;  
\item  the bound $ \fC_t \ll H^{r-1/2+o(1)}$ of Lemma~\ref{lem:Energy Fp}  if $  t =2$; 
 \item  the bound $ \fC_t \ll H^{r-1+o(1)}$ of Lemma~\ref{lem:Energy Fp}   if $ t \ge 3$. 
\end{itemize}

Hence, instead of the bounds~\eqref{eq:t in [0,1]},  \eqref{eq:t = 2},  
and~\eqref{eq:t ge 3} we now have the following
estimates.
 
 For $t \in\{0,1\}$ the total contribution $\fL_{0,1}$ satisfies
\begin{equation}\label{eq:t in [0,1] p}
\fL_{0,1} \ll  H^{m +nr - r} . 
\end{equation} 
(exactly as~\eqref{eq:t in [0,1]}). 

For $t = 2$ the total contribution $\fL_{2}$ now satisfies 
\begin{equation}\label{eq:t = 2 p}
\begin{split}
\fL_{2} &\le H^{r^2+o(1)}H^{2(m-r)} \( H^{r-1/2}\)^{n-r}  \\
& = H^{2m +n(r-1/2) - 3r/2+o(1)} . 
\end{split}
\end{equation}

For $t \ge 3$,  the total contribution $\fL_{\ge 3}$ satisfies 
\begin{equation}
\label{eq:t  ge 3 p}
\begin{split}
\fL_{\ge 3}&\le H^{r^2+o(1) } \sum_{t =3}^{r}  H^{t(m-r)} \( H^{r-1}\)^{n-r}  \\
& \le 
H^{r^2+o(1) } H^{r(m-r)} \( H^{r-1}\)^{n-r}    \\
& = H^{rm +n(r-1) - r(r-1) +o(1)} . 
\end{split}
\end{equation}

Substituting the bounds~\eqref{eq:t in [0,1] p},  \eqref{eq:t = 2 p},  and~\eqref{eq:t ge 3 p} in the inequality 
\[
L_{\vf, r}(H,p) \ll L_{\vf, r}^*(H,p) \le \fL_{0,1}  + \fL_{2}  +   \fL_{\ge 3}, 
\]
where $L_{\vf, r}^*(H,p)$ is defined as in Section~\ref{sec:T1 2}, we conclude the proof.

\subsection{Proofs of Theorems~\ref{thm: sing poly matr fij Fp} and~\ref{thm: imm poly matr fij Fp}}

Recall that by Lemma~\ref{lem:detpolyirr}, the polynomial
\[
 D\( \(X_{i,j}\)_{1\le i,j\le n}\) =\det\(f_{i,j}(X_{i,j})\)_{1\le i,j\le n},  
\] 
is absolutely irreducible over $\Fp$ for sufficiently large prime $p$. 
Hence by the  famous result of Lang and Weil~\cite{LaWe}, we have
\begin{equation}
\label{eq:LW}
T_D(p) = p^{n^2-1} + O(p^{n^2-3/2}),
\end{equation}
where we recall that $T_D(p)$ is the number of solutions to the congruence
\[
D\((x_{i,j})_{1\le i,j\le n}\)\equiv 0 \pmod p
\]
with unrestricted variables from $\F_p$.

To apply Lemma~\ref{lem:box}, we need to show that the algebraic variety $\cV=\cZ(D)$, where $\cZ(D)$ denotes the set of zeros of $D$ in $\C^{n^2}$, is not contained in any hyperplane. Let us assume that $\cV\subseteq \cH$, where $\cH=\cZ(L)$ is the hyperplane defined by an affine polynomial $L \in \C[X_{1,1},\ldots,X_{n,n}]$. Since both $D$ (by  Lemma~\ref{lem:detpolyirr}) and $L$ are absolutely irreducible, the ideals generated by them, $I$ and $J$, respectively, are radical ideals (in fact, they are prime). 
 Therefore, the inclusion $\cV\subseteq \cH$, implies, via the (strong) Nullstellensatz, the inclusion of ideals $J\subseteq I$, see for example~\cite[Theorems~6 and~7, Chapter~4]{CoxLitShe}. 
Thus, we have $L \in I$, which means that $D$ is a divisor of $L$ in $\C[X_{1,1},\ldots,X_{n,n}]$, which is a contradiction. Hence, $\cV$ is not contained in any hyperplane.  

Thus, since by Lemma~\ref{lem:detpolyirr}, $D$ is also  irreducible over $\C$, the result follows by an application of Lemma~\ref{lem:box} and~\eqref{eq:LW}, after one verifies that 
\[
 \frac{p^{n^2-3/2}}{p^{n^2}} H^{n^2} = H^{n^2} p^{-3/2} \le   H^{n^2-2}p^{1/2} 
\] 
 for $H \le p$, which give Theorem~\ref{thm: sing poly matr fij Fp}. 
 
The proof of Theorem~\ref{thm: imm poly matr fij Fp} is fully analogous, except using 
  polynomials $f_{i,j}(X_{i,j})\in \Z[X_{i,j}]$, $i,j=1,\ldots,n$ of the same degree, we are now able to use 
 Lemma~\ref{lem:lcminorsirr}  instead of  Lemma~\ref{lem:detpolyirr}. Indeed, we also note that the same argument as above, applied with the nonlinear polynomial 
 \[
\sum_{\sigma \in \cS_n} \chi(\sigma) \prod_{i=1}^n f_{i,\sigma(i)}(X_{i,\sigma(i)}), 
\]
which is irreducible by Lemma~\ref{lem:lcminorsirr}, shows that its zero set in $\C^{n^2}$ is not contained in any hyperplane.

\section{Further improvements and generalisations}
\label{sec:improve}

First we recall that Remark~\ref{rem:Better I_k} outlines a possibility for 
deriving slightly stronger bounds. 

 Next, we observe that for polynomials of  large degree, one can also use the bound
$$
 I_{6}(f,H)     \le H^{3+ o(1)}\(H^{1/3} + H^{2/\sqrt{d} + 1/(d-1)}\)
$$
of Browning~\cite[Theorem~2]{Brow}, see also~\cite{BrHB0}, which is stronger than the bound on $ I_{6}(f,H)$ 
in Lemma~\ref{lem:Small k} for $d \ge 20$. Then, using the  H{\"o}lder inequality one can also 
estimate $ I_{5}(f,H) \le I_{6}(f,H)^{5/6}$.

We also note that in the case when the polynomials $f_{i,j}$ are 
monomials $f_{i,j}(X_{i,j}) = X_{i,j}^d$,  $1\le i\le m$, $1 \le j\le n$,  of  the same degree $d\ge 3$ then 
using~\cite[Corollary~14.7]{Wool} one can for some parameters obtain a stronger version of  Lemma~\ref{lem:Wool}
and thus of Theorem~\ref{thm: rank poly matr fij Z 1}. This corresponds to the scenario of~\cite{BlLi}. 

Lemma~\ref{lem:lcminorsirr}  also allows us to get an asymptotic formula of the type of
 Theorem~\ref{thm: sing poly matr fij Fp} for the number of singular matrices 
 of the form
 \[
 \begin{pmatrix} f_{1,1}\(x_{1,1}\) & \ldots &  f_{1,n}\(x_{1,n}\) \\
 \vdots & \ldots& \vdots\\
 f_{n-1,1}\(x_{n-1,1}\) & \ldots &  f_{n-1,n}\(x_{n-1,n}\) \\
 a_1 & \ldots & a_n
 \end{pmatrix}, 
\]
with integers $x_{i,j} \in [-H,H]$, $1\le i\le n$, $1 \le j\le n-1$, for a fixed 
vector $\va = (a_1, \ldots, a_n) \in \F_p^n$.   This has an interpretation as the number of polynomial vectors containing $\va$ in their span.

Furthermore, using several recent results coming from additive combinatorics, such as
of Bradshaw,  Hanson and   Rudnev~\cite[Theorem~5]{BrHaRu} and of Mudgal~\cite[Theorem~1.6]{Mud}, 
one can obtain analogues of our results for matrices with elements of arbitrary 
but sufficiently quickly growing sequences.

In the case of matrices defined over $\F_p$, we note that if the polynomials $\vf$ in 
Theorem~\ref{thm: rank poly matr fij Fp}
are fixed of degree at most $e$ and defined over $\Z$ then for $H \le c(\vf) p^{1/e}$ with some constant $c(\vf)>0$ depending only on $\vf$, 
one can get results of the same strength as in Theorem~\ref{thm: rank poly matr fij Z 1}.
Indeed, in this case the congruence
\[
\sum_{i=1}^k f(x_i) \equiv \sum_{i=1}^k f(y_i) \pmod p, \qquad  -H  \le x_i, y_i \le H, \ i =1, \ldots, k, 
\]
whose number of solutions is given by $J_{2k}(f,H,p)$,
becomes an equation. Therefore,  we now have $J_{2k}(f,H,p) = I_{2k}(f,H)$ and hence the bounds of Section~\ref{sec: sol box Z}
apply.

 We also note that  by~\cite[Lemma~6]{Go-PSh} one can multiply all coefficients of 
\[
f(X) = a_e X^e + \ldots + a_1X + a_0 \in \Z[X]
\]  
by some 
integer $\lambda \not \equiv 0 \pmod p$ such that their smallest by absolute value residues $b_j \equiv \lambda a_j \pmod p$ satisfy 
\[
b_j \ll p^{1-2j/(e(e+1))}, \qquad j=0, \ldots, e.
\] 
Hence  for $H \le c(e) p^{2/(e(e+1))}$ 
with some constant $c(e)$ depending only on $e$, we have 
$J_{2k}(f,H,p) = I_{2k}(g,H)$, where 
\[
g(X) = b_e X^e + \ldots + b_1X + b_0 \in \Z[X].
\] 
It remains to 
recall that many results obtained via the determinant method, such as~\cite{Pila}, are uniform
with respect to the polynomials involved.

Finally, we mention that in the same fashion as Theorems~\ref{thm: rank poly matr fij Z 1} and~\ref{thm: rank poly matr fij Fp}, one may obtain similar results allowing the polynomials $f_{i,j}(X_{i,j})$ in the matrix  
 $\vf $ to take values from sets $\cA\subseteq \Z$, or $\F_p$, of cardinality $A$, with small sum 
 set $\cA+\cA = \{a + b:~a,b\in \cA\}$. That is,  if $\#(\cA+\cA) \leq KA$ with $K= A^{1-\varepsilon}$, for some $\varepsilon>0$.

 For instance, given a quadratic polynomial $f\in \F_p[X]$, for sets $\cA\subset \F_p$, with $A\leq p^{2/3}$ and $\#(\cA+\cA)\ll A$, it follows from~\cite[Lemma~2.10]{ShkShp} that the number of solutions to the equation
 \begin{align*}
f(u) + f(v) +f(w) & = f(x) + f(y) + f(z), \\(u,v,w,x&,y,z) \in \cA^6,
\end{align*}
is $O\(A^{5-1/2}\)$. See also~\cite[Corollaries~2.10 and~2.11]{ShkShp}. 
We further point out that the relevant bounds of~\cite{ShkShp} also hold for subsets of $\Z$, due to the generality of their underlying result from~\cite{Rud}. Perhaps the approach of~\cite{KMS} can also be adapted to equations  of the above type with 6 or more variables.

We have already mentioned possible generalisations of our results from matrices with polynomial entries to 
matrices with elements coming from sequences with some additive properties. Another possible generalisation
is for sequences with  some multiplicative properties. First we recall that Alon and Solymosi~\cite[Theorem~1]{AlSol} 
have shown that matrices with entries from finitely generated subgroups of $\C^*$ have a rank growing with their 
dimension $n$. Using bounds on the number of solutions to so-called {\it $S$-unit equations\/}, see, for example,~\cite[Theorem~6.2]{AmVia} and our argument, 
one can obtain various counting versions of~\cite[Theorem~1]{AlSol}. 
 
\section*{Acknowledgements}

The authors are very grateful to 
Valentin Blomer and Junxian Li for useful comments on the preliminary  versions
of this paper and to Akshat Mudgal for  several  important suggestions and references, which 
helped to improve  some of the results.

During the preparation of this work, the authors were   supported in part by the  
Australian Research Council Grants  DP200100355 and DP230100530.

 \end{document}